\documentclass[12pt]{amsart}
\usepackage{amsmath}
\usepackage{amsfonts}
\usepackage{amsthm, upref}
\usepackage{graphicx}
\usepackage{enumerate}
\usepackage[usenames, dvipsnames]{color}

\usepackage{ifthen}
\usepackage{tikz}
\usepackage{appendix}
\usepackage{verbatim}
\usetikzlibrary{decorations.pathmorphing}
\usetikzlibrary{calc}
\usepackage{hyperref}

\usepackage{tikz}
\usepackage{amsfonts}
\usepackage{enumerate}
\usepackage[margin=0.88in]{geometry}
\usepackage{hyperref}
\usepackage{verbatim}

\allowdisplaybreaks[4]

\newcommand{\MZ}{\mathbb{Z}}

\newcommand{\BR}{\mathbb{R}}

\newcommand{\SL}{\sum\limits}

\newcommand{\al}{\alpha}

\newcommand{\be}{\gamma}
\newcommand{\ga}{\gamma}

\newcommand{\CF}{\mathcal F}
\newcommand{\CH}{\mathcal H}

\newcommand{\MP}{\mathbb{P}}
\newcommand{\MQ}{\mathbb{Q}}

\newcommand{\Oa}{\Omega}

\newcommand{\si}{\sigma}

\newcommand{\pa}{\partial}
\renewcommand{\phi}{\varphi}

\newcommand{\Ra}{\Rightarrow}

\newcommand{\EE}{\mathbf{R}}

\newcommand{\norm}[1]{\left\lVert#1\right\rVert}
\renewcommand{\comment}[1]{}

\newcommand{\mP}{\mathbf{p}}

\newcommand{\md}{\mathrm{d}}

\newcommand{\fn}{\mathbf{n}}

\DeclareMathOperator{\diag}{diag}

\renewcommand{\comment}[1]{}
\newcommand{\eq}{\begin{equation}}
\newcommand{\en}{\end{equation}}

\begin{document}

\theoremstyle{plain}
\newtheorem{thm}{Theorem}[section]
\newtheorem{lemma}[thm]{Lemma}
\newtheorem{prop}[thm]{Proposition}
\newtheorem{cor}[thm]{Corollary}
\newtheorem{open}[thm]{Open Problem}

\theoremstyle{definition}
\newtheorem{defn}{Definition}
\newtheorem{asmp}{Assumption}
\newtheorem{notn}{Notation}
\newtheorem{prb}{Problem}

\theoremstyle{remark}
\newtheorem{rmk}{Remark}
\newtheorem{exm}{Example}
\newtheorem{clm}{Claim}

\title[TCI for Reflected Diffusions]{A Note on Transportation Cost Inequalities\\ for Diffusions with Reflections}  

\author{Soumik Pal}

\address{Department of Mathematics, University of Washington, Seattle}

\email{soumikpal@gmail.com}

\author{Andrey Sarantsev}

\address{Department of Mathematics and Statistics, University of Nevada, Reno}

\email{asarantsev@unr.edu}

\date{\today}


\thanks{S.P. was supported by NSF grant DMS 1612483  and A.S. was supported in part by the NSF grant DMS 1409434.} 

\begin{abstract}
We prove that reflected Brownian motion with normal reflections in a convex domain satisfies a dimension free Talagrand type transportation cost-information inequality. The result is generalized to other reflected diffusion processes with suitable drift and diffusion coefficients. We apply this to get such an inequality for interacting Brownian particles with rank-based drift and diffusion coefficients such as the infinite Atlas model. This is an improvement over earlier dimension-dependent results.
\end{abstract}

\maketitle

\thispagestyle{empty}

\section{Introduction and Main Results}  






Consider a metric space $(E, \rho)$. Fix $p \in [1, \infty)$. For any pair of Borel probability measures $\MP$ and $\MQ$ on $E$, define the Wasserstein distance of order $p$ as 
$$
W_p(\MP, \MQ) := \inf\limits_{\pi\in \Pi}\left[\iint\rho^p(x, y)\md\pi(x, y)\right]^{1/p},
$$
where the $\inf$ is taken over the set $\Pi$ of all {\it couplings} of $\MP$ and $\MQ$ (i.e., measures on $E\times E$ with marginal distributions $\MP$ and $\MQ$). Here and throughout, when we write $\mu(f)$, where $\mu$ is a probability measure and $f$ is a $\mu$-integrable function, we mean the expectation of $f$ with respect to $\mu$. The {\it relative entropy} $\CH(\MQ\mid\MP)$ of $\MQ$ with respect to $\MP$ is defined as
$$
\CH(\MQ\mid\MP) := 
\MQ\left[\log\frac{\md\MQ}{\md\MP}\right] = \MP\left[\frac{\md\MQ}{\md\MP}\log\left(\frac{\md\MQ}{\md\MP}\right)\right],\quad \mbox{if}\; \MQ \ll \MP,
$$
and $\CH(\MQ\mid\MP) = +\infty$ otherwise. 

\begin{defn} A Borel probability measure $\MP$ satisfies the {\it transportation-cost information (TCI) inequality of order} $p$ with constant $C > 0$ (we write: $\MP \in T_p(C)$) if for every Borel probability measure $\MQ$ on $E$ we have:
\begin{equation}
\label{eq:TCI}
W_p(\MP, \MQ) \le \sqrt{2C\CH(\MQ\mid\MP)}.
\end{equation}
\end{defn}

TCI is an example of the vast gallery of various inequalities linking transportation cost, relative entropy, and Fisher information. It is impossible to do justice to the enormous literature and its many uses. We refer the reader to an excellent survey by Gozlan and 
L\'eonard \cite{Gozlansurvey} and the recent book \cite{NewBook} by Boucheron et al.  Talagrand studied concentration for product spaces in \cite{Talagrand3, Talagrand5}. His idea is that a function of many variables which is Lipschitz in any variable but does not depend much on any single variable is close to a constant.  A particularly useful application of TCI inequalities, such as above (for $p\ge 1$), is to prove Talagrand type Gaussian concentration. See  the original article by Talagrand \cite{Talagrand5}, as well as Marton's derivation by using TCI inequality in \cite{Marton}. See also \cite{Bobkov1, C3, Otto}  on relation between TCI and log-Sobolev inequalities. 

 In \cite{Talagrand5}, Talagrand proved that a standard Gaussian measure on $\BR^d$ satisfies $T_2(C)$ with $C = 1$. Afterwards, TCI inequalities were established for discrete-time Markov chains, \cite{Marton1, Paulin, Samson}; for discrete-time stationary processes, \cite{Marton2}; for stochastic ordinary differential equations driven by Brownian motion \cite{LogConcave, Djellout, Pal, Ustunel} and by more general noise \cite{MoroccoFBM, Saussereau, Riedel}; for stochastic partial differential equations, \cite{Morocco, Davar, RDE}, and for neutral stochastic equations (which depend on past history) \cite{Neutral, Delay, MoroccoFBM}. Applications include model selection in statistics \cite{Model}, risk theory \cite{Lacker}, order statistics \cite{Order}, information theory \cite{Bobkov3, Sason}, and randomized algorithms \cite{Algorithm}.



In this paper, $\MP$ and $\MQ$ will represent laws of reflecting diffusion processes seen as probability measures on the set of continuous paths equipped with the uniform norm. Specifically we prove TCI inequalities for a certain class of interacting Brownian particle systems, called {\it competing Brownian particles}, with each particle moving as a Brownian motion on the real line with drift and diffusion coefficients dependent on the current rank of this particle relative to other particles. These systems were constructed in \cite{BFK} as a model for financial markets; see also \cite{CP2010, FernholzBook}. Our inequalities are {\it dimension-free:}  that is, the constant  $C$ is independent of the number of particles. This allows us to extend the inequality to infinite competing particle systems such as the infinite Atlas model \cite{CDSS, DemboTsai, KolliShkol, PalPitman}. This is an improvement over the dimension-dependent inequalities in papers \cite{Pal, PalShkolnikov} where applications of such inequalities can be found. See also \cite{IPS2013, JM2008} on Poincar\'e inequalities for competing Brownian particles, and \cite{4people} on large deviations for these particle systems. 

The result for competing particles is a particular case of a general TCI inequality for normally reflected diffusion processes in convex domains. Reflected diffusions are defined as continuous-time stochastic processes in a certain domain $D \subseteq \BR^d$. As long as such process is in the interior, it behaves as a solution of a stochastic differential equation (SDE). As it hits the boundary, it is reflected back inside the domain. The simplest case is a {\it reflected Brownian motion}, which behaves as a Brownian motion inside the domain. 

Dimension-free TCI inequalities are remarkable. Most known examples are in the case of product measures which utilize tensorization property of the entropy and the cost. Our examples are far from product measures since they involve motion of particles interacting with one another. Hence, dimension-free TCI inequalities in this context seem interesting. The proof, however, does not require much beyond existing machinery. Our novel contribution is essentially a single observation made in \eqref{eq:nonincreasing-1}.

\subsection{Notation} We denote by $a\cdot b = a_1b_1 + \ldots + a_Nb_N$ the dot product of vectors $a, b \in \BR^N$. The Euclidean norm of a vector $a$ is denoted by $\norm{a} := \left[a\cdot a\right]^{1/2}$. The matrix norm of a matrix $A$ is defined as $\norm{A} := \max_{\norm{x} = 1}\norm{Ax}$. We let $\BR_+ := [0, \infty)$. We denote the space $C\left([0, T], \BR^N\right)$ of continuous functions $[0, T] \to \BR^N$ with the sup-norm $\norm{x} := \sup_{0 \le t \le T}\norm{x(t)}$. We prove TCI inequality where the Wasserstein-$2$ transportation cost is measured in this norm.

\subsection{TCI inequalities for competing Brownian particles}  Fix an integer $N \ge 2$. For any vector $x = (x_1, \ldots, x_N) \in \BR^N$, there exists a unique {\it ranking permutation}: a one-to-one mapping $\mP_x : \{1, \ldots, N\} \to \{1, \ldots, N\}$, with the following properties:
\begin{enumerate}[(a)]
\item $x_{\mP_x(i)} \le x_{\mP_x(j)}$ for $1 \le i < j \le N$;
\label{defn:ranks}
\item if $x_{\mP_x(i)} = x_{\mP_x(j)}$ for $1 \le i < j \le N$, then $\mP_x(i) < \mP_x(j)$. 
\label{defn:ranks2}
\end{enumerate}
That is, $\mP_x$ arranges the coordinates of $x$ in increasing order, with ties broken by the increasing order of the index (or, \textit{{name}}) of the coordinates that are tied. 

\smallskip

Take a filtered probability space $(\Oa, \CF, (\CF_t)_{t \ge 0}, \MP)$, with the filtration satisfying the usual conditions and supporting an $N$-dimensional Brownian motion $W = (W_1, \ldots, W_N)$. Fix constants $g_1, \ldots, g_N \in \BR$ and $\si_1, \ldots, \si_N > 0$. 

\begin{defn}
Consider a continuous adapted process $X(t) = (X_1(t), \ldots, X_N(t)),\ t \ge 0$. Let $\mP_t = \mP_{X(t)}$. We say that $\mP_t^{-1}(i)$ is the {\it rank} of particle $i$ at time $t$, and $\mP_t(k)$ is the {\it name} of the $k$th ranked particle at time $t$. Then the following system of SDE:
\begin{equation}
\label{eq:mainSDE}
\md X_i(t) = \SL_{k=1}^N1(\mP_t(k) = i)\left(g_k\md t + \si_k\md W_i(t)\right),
\end{equation}
for $i = 1, \ldots, N$, defines a finite system of $N$ competing Brownian particles with drift coefficients $g_1, \ldots, g_N$ and diffusion coefficients $\si_1^2, \ldots, \si_N^2$. Let $Y_k(t) := X_{(k)}(t) := X_{\mP_t(k)}(t)$ be the position of the $k$th ranked particle, and let $Z_k(t) := Y_{k+1}(t) - Y_k(t)$ be the gap between the $k$th and $(k+1)$st ranked particles. The local time $L_{(k, k+1)} = (L_{(k, k+1)}(t), t \ge 0)$ of collision between $k$th and $k+1$st ranked particles is defined as the local time of the continuous semimartingale $Z_k$ at zero. The process
$L(t) = \left(L_{(1, 2)}(t), \ldots, L_{(N-1, N)}(t)\right)$ is called the vector of local times.
\label{defn:CBP}
\end{defn}
From \cite{Bass1987}, this system exists in the weak sense and is unique in law. Strong existence and pathwise unqiueness are proved under the following assumptions, \cite[Theorem 2]{IKS2013}, \cite[Theorem 1.4]{MyOwn3}. 
\begin{equation}
\label{eq:convex-diffusion}
\sigma^2_n \ge \frac12\left(\sigma^2_{n-1} + \sigma^2_{n+1}\right),\, \mbox{if}\;  
1 < n < N.
\end{equation}
Similar infinite systems can be defined for $N = \infty$; then we assume that the vector $X(t) = (X_i(t))_{i \ge 1}$ is {\it rankable}; that is, for every $t\ge 0$ there exists a unique permutation $\mP_{X(t)}$ of $\mathbb{N} := \{1, 2, \ldots\}$ which satisfies conditions (\ref{defn:ranks}) and (\ref{defn:ranks2}). They were introduced in \cite{PalPitman}. See \cite[Theorem 3.1]{MyOwn6} for weak existence and uniqueness in law under an assumption on initial conditions, 
\begin{equation}
\label{eq:initial-asmp}
\lim\limits_{n \to \infty}X_n(0) \to \infty\quad \mbox{a.s., and}\quad \sum\limits_{n=1}^{\infty}e^{-\alpha X_n^2(0)} < \infty\quad \mbox{for all}\quad \alpha > 0,
\end{equation}
and the following assumptions on drift and diffusion coefficients:
\begin{equation}
\label{eq:coeff-asmp}
g_{n}  = g_{n_0}\quad \mbox{and}\quad \sigma_n = \sigma_{n_0}\quad \mbox{for all}\quad n \ge n_0.
\end{equation}
See  \cite[Theorems 1, 2]{IKS2013} and \cite[Theorem 5.1, Remark 8]{MyOwn6} for strong existence and pathwise uniqueness: We need~\eqref{eq:convex-diffusion} in addition to~\eqref{eq:initial-asmp} and~\eqref{eq:coeff-asmp}. {\it Two-sided infinite systems}, indexed by $i \in \MZ$, were introduced in \cite{MyOwn11}. The proofs and the results from this paper carry over to that set-up as well. 




\begin{thm} (a) For an $N \in \mathbb{N} \cup \{\infty\}$ assume that the drift and diffusion coefficients satisfy the following conditions: $g_1 \ge g_2 \ge \ldots$, and $\si_1 = \si_2 = \ldots = 1$. For the case of an infinite system, assume in addition~\eqref{eq:initial-asmp} and~\eqref{eq:coeff-asmp}. Then for every finite $k \le N$, the distribution of $X = (X_1, \ldots, X_k)$ on $C([0, T], \BR^k)$ satisfies $T_2(C)$ with $C = T$.
\label{thm:CBP}

\smallskip

(b) Assume weak existence and uniqueness in law. For an $N \in \mathbb{N}\cup\{\infty\}$, and a  finite $k \le N$, the vector of ranked particles $Y = (Y_1, \ldots, Y_k)$ satisfies $T_2(C)$ on $C([0, T], \BR^k)$ with $C = T\sup\limits_{m \ge 1}\si_m^2$. 
\end{thm}



Theorem~\ref{thm:CBP} (a) follows from results from \cite{LogConcave}. The more non-trivial Theorem~\ref{thm:CBP} (b) is based on Theorem~\ref{thm:main} below, which is the main result of this paper. This is a general result that says that normally reflected Brownian motion in a convex domain satisfies a dimension-free TCI inequality as described below. It turns out that the vector of ranked particles $Y$ is a particular case of such normally reflected Brownian motion in  a wedge $\{y = (y_1, \ldots, y_N)\mid y_1 \le \ldots \le y_N\}$. 

\smallskip

Fix $d \ge 2$, the dimension. In this article a {\it domain} in $\BR^d$ is the closure of an open connected subset. We consider only convex domains. Following \cite{Tanaka1979}, we do not impose any additional smoothness conditions on such domain. For every $x \in \partial D$, we say that a unit vector $y \in \BR^d$ is an {\it inward unit normal vector} at point $x$, if  
\begin{equation}
\label{eq:inward-normal}
z \in D\quad \mbox{implies}\quad (z - x)\cdot y \ge 0.
\end{equation}
The set of such inward unit normal vectors at $x$ is denoted as $\mathcal N(x)$. The most elementary example of this is a $C^1$ domain $D$; that is, with boundary $\pa D$ which can be locally (after a rotation) parametrized as a graph of a $C^1$ function. Then there exists a unique inward unit normal vector $\fn(x)$ at every point $x \in \pa D$, and $\mathcal N(x) = \{\fn(x)\}$. 


A more complicated example is a {\it convex piecewise smooth domain}. Fix $m \ge 1$, the number of faces. Take $m$ domains $D_1, \ldots, D_m$ in $\BR^d$ which are $C^1$ and convex. Let $D = \bigcap_{i=1}^mD_i$. Assume $D \ne \cap_{j \ne i}D_j$ for every $i = 1, \ldots, m$; that is, each one of $m$ smooth domains is essential. Assume also that for each $i = 1, \ldots, m$, $F_i := \pa D\cap\pa D_i$ is a manifold of codimension $1$ and has nonempty relative interior. Then $D$ is called a {\it convex piecewise smooth domain} with $m$ {\it faces} $F_1, \ldots, F_m$, and $\pa D = \cup_{i=1}^mF_i$.  For every $x \in F_i$, define the {\it inward unit normal vector} $\fn_i(x)$ to $\pa D_i$ at this point $x$, pointing inside $D_i$. For a point $x \in \pa D$ on the boundary, if $I(x) = \{i = 1, \ldots, m\mid x \in F_i\}$, then 
$$
\mathcal N(x) = \Bigl\{\sum\limits_{i \in I}\al_i\fn_i(x)\mid \al_i \ge 0,\quad i \in I;\quad \sum\limits_{i \in I}\al_i^2 = 1\Bigr\}.
$$



\begin{defn} For a vector field $g : \BR_+\times D \to \BR^d$ and a $z_0 \in D$, consider the following equation:
\begin{equation}
\label{eq:SDE-main}
Z(t) = z_0 + W(t) +  \int_0^tg(s, Z(s))\,\md s + \int_0^t\fn(s)\,\md \ell(s),\ \ t \ge 0.
\end{equation}
Here, $Z : \BR_+ \to D$ is a continuous adapted process, $W$ is a $d$-dimensional Brownian motion with zero drift vector and constant, symmetric, positive definite $d\times d$ covariance matrix $A$, starting from the origin. For every $t \ge 0$, $\fn(t) \in \BR^d$ is a unit vector, and $\fn : \BR_+ \to \BR$ is a measurable function. The function $\ell : \BR_+ \to \BR$ is continuous, nondecreasing, and can increase only when $Z(t) \in \pa D$; for such $t$, $\fn(t) \in \mathcal N(Z(t))$. The solution $Z$ of \eqref{eq:SDE-main} is called a {\it reflected diffusion} in $D$, with {\it drift} $g$ and (constant) {\it diffusion matrix} $A$, starting from $z_0$. If $g$ is a constant (does not depend on $x$ and $t$), then we call $Z$ a {\it reflected Brownian motion} (RBM) in $D$ with {\it drift} $g$ and {\it diffusion matrix} $A$. 
\end{defn}

\begin{asmp} For an integrable function $F : [0, T] \to \BR$, we have:
\begin{equation}
\label{eq:monotone-F}
\left(g(t, x) - g(t, y)\right)\cdot(x - y) \le \norm{x-y}^2F(t),\ t \in [0, T],\ x, y \in D,
\end{equation}
\label{asmp:main}
\end{asmp}

\begin{rmk} Under Assumption~\ref{asmp:main}, the equation~\eqref{eq:SDE-main} has a pathwise unique strong solution on time horizon $[0, T]$. This is proved similarly to  \cite[Theorem 4.1]{Tanaka1979}. 
\label{rmk:strong-existence}
\end{rmk}


We start by proving that, under Assumption \ref{asmp:main}, the reflected diffusion satisfies a dimension-free $T_2(C)$. Our proof follows existing ideas in \cite{Pal, LogConcave} for non-reflected diffusions with one notable observation to handle the reflection. Consider an SDE in $\BR^d$ without reflection $\md X(t) = \md W(t) + g(t, X(t))\,\md t$, for some drift vector field $g$ defined on $[0, T]\times \mathbb{R}^d$, which satisfies {\it contraction condition:}
\begin{equation}
\label{eq:contraction}
(g(t, x) - g(t, y)) \cdot (x - y) \le 0,\quad \mbox{for all}\quad x, y \in \BR^d,\, t \in [0, T].
\end{equation}
It is shown in \cite{LogConcave} that under condition~\eqref{eq:contraction}, the distribution of $X$ in the space $C([0, T], \BR^d)$ satisfies $T_2(C)$ with $C = T$. Our main observation in this article is that for a reflected diffusion in a convex domain $D$, the reflection term $\fn(t)\,\md\ell(t)$ plays the role of such drift. 




For the next result, take a convex domain $D$, fix time horizon $T > 0$, and let $\MP$ denote the law of the reflected diffusion in $C([0, T], \BR^d)$ with drift vector field $g$, starting from $x \in D$. 


\begin{thm} Under Assumption~\ref{asmp:main},  $\MP \in T_2(C)$, with the constant $C$ given by
\begin{equation}
\label{eq:constant-minus}
C := \norm{A}\sup\limits_{0 \le t \le T}\int_0^t\exp\left(2\int_s^tF(u)\,\md u\right)\,\md s.
\end{equation}
\label{thm:main}
\end{thm}

If $F(t) = \gamma$ is a constant, then we can calculate
$$
C= 
\begin{cases}
\norm{A}\frac{e^{2\gamma T} - 1}{2\gamma},\quad \gamma \neq 0;\\
\norm{A}T, \quad \gamma=0.
\end{cases}
$$
This gives us the following corollary. 



\begin{cor} The law of an RBM in a convex domain $D$ with constant drift and constant diffusion matrix $A$ satisfies $T_2(C)$ on $C([0, T], \BR^d)$ with $C = T\norm{A}$. 
\label{cor:RBM}
\end{cor}

\section{Proofs} 


\begin{proof}[Proof of Theorem~\ref{thm:main}]
The established method for proving TCI inequality for diffusions is by using Girsanov theorem. We explain the main idea behind this line of argument. More details can be found in \cite{Djellout, Pal, Ustunel}. We assume for simplicity that $A = I_d$. At the end of this subsection, we shall explain what to do for general $A$. Take a filtered probability space $(\Oa, \CF, (\CF_t)_{0 \le t \le T}, \mathbf R)$, with the filtration satisfying the usual conditions and generated by a $d$-dimensional Brownian motion $W = (W(t),\, 0 \le t \le T)$. By Assumption 1, on this space we can construct a solution $X$ to the equation~\eqref{eq:SDE-main} driven by the Brownian motion $W$. We view $X$ as a random element of the space $C([0, T], \BR^d)$ with law $\MP$. On $C([0, T], \BR^d)$, take any probability measure $\MQ \ll \MP$. Let $\overline{\mathbf R} \ll \mathbf R$ be another probability measure on the space $(\Oa, \CF)$, defined through its Radon-Nikodym derivative:
\begin{equation}
\label{eq:fundamental-change}
\frac{\md\overline{\mathbf R}}{\md\mathbf R} := \frac{\md\MQ}{\md\MP}(X).
\end{equation}
The next lemma is taken from \cite[Proof of Theorem 5.6]{Djellout}. See also related papers \cite{Feyel-Ustunel, Ustunel}.

\begin{lemma} There exists an $\left(\CF_t\right)$-adapted process $\be = (\be_t,\ 0 \le t \le T)$ such that, $\overline{\mathbf R}$-almost surely $\int_0^T\norm{\be_t}^2\md t < \infty$, and the following process is a standard $d$-dimensional $\left((\CF_t), \overline{\mathbf R}\right)$-Brownian motion:
\[
B(t) = W(t) - \int_0^t\be_s\md s,\quad 0 \le t \le T.
\]
Moreover,  
\begin{equation}
\label{eq:entropy-derivation}
\CH(\overline{\mathbf R}\mid\mathbf R) = \frac{1}{2} \overline{\EE}\left[ \int_0^T\norm{\be_t}^2\md t\right].
\end{equation}
\label{lemma:change-of-measure}
\end{lemma}

From~\eqref{eq:fundamental-change}, the relative entropy of $\MQ$ with respect to $\MP$ is given by the same formula~\eqref{eq:entropy-derivation}:
\begin{equation}
\label{eq:entropy}
\CH(\MQ\mid\MP) = \CH(\overline{\mathbf R}\mid\mathbf R) = \frac{1}{2} \overline{\EE}\left[ \int_0^T\norm{\be_t}^2\md t\right].
\end{equation}
The law of $X$ on the probability space 
\begin{equation}
\label{eq:new-probability-space}
\left(\Oa, \CF, (\CF_t)_{0 \le t \le T}, \overline{\mathbf R}\right)
\end{equation}
is $\MQ$ instead of $\MP$. Let $X'$ be the solution of~\eqref{eq:SDE-main} on the space~\eqref{eq:new-probability-space} with Brownian motion $B$ instead of $W$. This solution exists and is unique by Remark~\ref{rmk:strong-existence}. Then $X'$ has law $\MP$. Hence, on that probability space~\eqref{eq:new-probability-space} we now have two processes $(X,X')$ such that $X'\sim \MP$ and $X \sim \MQ$, together with continuous adapted nondecreasing processes $\ell = (\ell(t),\, 0 \le t \le T)$ and $\ell' = (\ell'(t),\ 0 \le t \le T)$, starting from $0$, such that $\ell$ can increase only when $X \in \pa D$; similarly for $\ell'$, and
\begin{eqnarray}
X(t) &=& x + \int_0^tg(s, X(s))\,\md s + B(t) + \int_0^t\be_s\,\md s + \int_0^t\fn(s)\,\md\ell(s),\label{eq:X}\\
X'(t) &=& x + \int_0^tg(s, X'(s))\,\md s  + B(t) + \int_0^t\fn'(s)\,\md\ell'(s). \label{eq:X'}
\end{eqnarray}
From~\eqref{eq:X} and~\eqref{eq:X'}, we get:
\begin{align}
\label{eq:X-X-diff}
\begin{split}
X(t) - X'(t) = \int_0^t\left[g(s, X(s)) - g(s, X'(s)) + \be_s\right]\md s +\int_0^t\left[\fn(s)\,\md\ell(s) - \fn'(s)\,\md\ell'(s)\right].
\end{split}
\end{align}
We claim
\begin{equation}
\label{eq:goal}
\overline{\EE}\left[\norm{X-X'}^2\right] \le C\cdot\overline{\EE}\Bigl[\int_0^T\norm{\be_t}^2\,\md t\Bigr] = 2C\CH(\MQ\mid\MP).
\end{equation}
Since $(X',X)$ is a coupling of the $(\MP, \MQ)$,~\eqref{eq:goal} gives an upper bound on the $\mathcal W_2$-distance, and hence $T_2(C)$.   For general (constant) diffusion $A$, let $A^{1/2}$ refer to its positive definite square root. Write $A^{1/2}\ga_s$ instead of $\be_s$ in~\eqref{eq:X},~\eqref{eq:X-X-diff}, and subsequent places; then observe that $\norm{A^{1/2}\be_s}^2 \le \norm{A}\norm{\be_s}^2$. 



%


\smallskip

To prove ~\eqref{eq:goal}, define
\begin{equation}
\label{eq:norm-Y}
Y(t) := \norm{X(t) - X'(t)}\,;\quad \mbox{then}\quad Y^2(t) = (X(t) - X'(t))\cdot(X(t) - X'(t)).
\end{equation}
Since $X-X'$ is continuous and of finite variation, the same can be said of $Y$. Thus we can apply the classic chain rule ({\it not} It\^o's formula) to the process $Y^2(\cdot)$:
\begin{equation}
\label{eq:square}
\md Y^2(t) = 2(X(t) - X'(t))\cdot\md(X(t) - X'(t)) = 2Y(t)\,\md Y(t).
\end{equation}
Combining ~\eqref{eq:X-X-diff} and~\eqref{eq:square}, we get: 
\begin{align}
\label{eq:Ito-phi}
\begin{split}
\md Y^2(t)  & = 2(X(t) - X'(t))\cdot \left[g(t, X(t)) - g(t, X'(t))\right]\,\md t + 2(X(t) - X'(t))\cdot \be_t\,\md t \\ & + 2\fn(t)\cdot(X(t) - X'(t))\,\md\ell(t) - 2\fn'(t)\cdot(X(t) - X'(t))\,\md\ell'(t). 
\end{split}
\end{align}

The next remark is on differentials of continuous functions with bounded variation.

\begin{rmk} 
For two continuous functions $f_1, f_2 : [0, T] \to \mathbb R$ of bounded variation, we write $\md f_1(t) \le \md f_2(t)$ for all $t$ in a subinterval $I \subseteq [0, T]$, if $f_1(t) - f_1(s) \le f_2(t) - f_2(s)$ for all $s, t \in I,\, s < t$. This is equivalent to the following condition: for the signed measures $\mu_1$ and $\mu_2$ on $[0, T]$ defined by $\mu_i[0, t] = f_i(t) - f_i(0),\, t \in [0, T],\, i = 1, 2$, the measure $\mu_2-\mu_1$ is nonnegative on $[0,T]$; that is, $\mu_1(B) \le \mu_2(B)$ for any Borel set $B \subseteq [0, T]$.  
For continuous functions $f_1, f_2 : [0, T] \to \mathbb R$ of bounded variation, and for a continuous function $g : [0, T] \to [0, \infty)$, if $\mathrm{d}f_1(t) \le \mathrm{d}f_2(t)$, then $\mathrm{d}F_1(t) \le \mathrm{d}F_2(t)$, where $F_1, F_2$ are defined as follows: 
$$
F_i(t) = \int_0^tg(s)\,\mathrm{d}f_i(s),\ i = 1, 2.
$$
We can write this as $g(t)\,\mathrm{d}f_1(t) \le g(t)\,\mathrm{d}f_2(t)$. 
\label{rmk:rules}
\end{rmk}

Now comes the crucial observation: 
\begin{equation}
\label{eq:nonincreasing-1}
\fn(t)\cdot(X(t) - X'(t))\,\md\ell(t) \le 0,\quad \text{for all}\; t \ge 0.
\end{equation}
Indeed, $\ell(t)$ can grow only when $X(t) \in \pa D$, and in this case $X'(t) \in D$, and therefore $\fn(t)\cdot(X(t) - X'(t)) \le 0$ from~\eqref{eq:inward-normal}. Combine this with $\md\ell(t) \ge 0$ and get~\eqref{eq:nonincreasing-1}. Similarly, 
$\fn'(t)\cdot(X(t) - X'(t))\,\md\ell'(t) \ge 0$.
Also from~\eqref{eq:monotone-F}, we get that
\begin{equation}
\label{eq:monotone}
(g(t, X(t)) - g(t, X'(t))\cdot(X(t) - X'(t)) \le F(t)\norm{X(t) - X'(t)}^2 = F(t)Y^2(t). 
\end{equation}
Thus, from ~\eqref{eq:Ito-phi}, we get:
\begin{equation}
\label{eq:basic-inequality}
\md Y^2(t) \le \left(2F(t)\,Y^2(t) + 2 (X(t) - X'(t))\cdot\be_t\right)\md t  \le \left(2F(t)Y^2(t) + 2Y(t)\norm{\be_t}\right)\md t.
\end{equation}
Using~\eqref{eq:square}, we rewrite~\eqref{eq:basic-inequality} as
\begin{equation}
\label{eq:point-of-choice}
2Y(t)\,\md Y(t) \le 2Y(t)\left(F(t)Y(t) + \norm{\be_t}\right)\,\md t.
\end{equation}
Now, we claim that for $t \in [0, T]$, 
\begin{equation}
\label{eq:ineq-Y}
Y(t)  \le \int_0^t\norm{\be_s}\exp\left(\int_s^tF(u)\md u\right)\,\md s.
\end{equation}
For every $t \in [0, T]$, either $Y(t) = 0$, and then~\eqref{eq:ineq-Y} is immediate, or $Y(t) > 0$. In this second case, we prove~\eqref{eq:ineq-Y} as follows.  Since the function $Y$ is continuous, the set $I := \{t \in (0, T)\mid Y(t) > 0\}$ is open, therefore is a countable union of disjoint open intervals. On each such interval $(\al_1, \al_2)$, $Y(t)>0$. According to Remark~\ref{rmk:rules}, we can multiply~\eqref{eq:point-of-choice} by $Y^{-1}(t) > 0$:
\begin{equation}\label{eq:ptchoice2}
\,\md Y(t) \le F(t)Y(t)\,\md t + \norm{\be_t}\,\md t.
\end{equation}
We can rewrite \eqref{eq:ptchoice2} as
$
\md Y(t) - F(t)Y(t)\,\md t  \le \norm{\be_t}\,\md t.
$
Multiplying by an integrating factor, we get:
$$
\md \left(Y(t)\exp\left(-\int_{\al_1}^tF(s)\md s\right)\right) \le \norm{\be_t}\exp\left(-\int_{\al_1}^tF(s)\md s\right)\,\md t.
$$
Integrating with respect to $t$ over $[\al_1, t]$ and using the fact that $Y(\al_1) = 0$, we get:
$$
Y(t)\exp\left(-\int_{\al_1}^tF(s)\md s\right) \le \int_{\al_1}^t\norm{\be_s}\exp\left(-\int_{\al_1}^sF(u)\md u\right)\,\md s.
$$
Thus, for $t \in (\al_1, \al_2)$,  
\eq
\begin{split}
Y(t)&\exp\left(-\int_0^tF(s)\md s\right) = \exp\left(-\int_0^{\al_1}F(s)\md s\right) Y(t)\exp\left(-\int_{\al_1}^tF(s)\md s\right)\\ 
& \le \int_{\al_1}^t\norm{\be_s}\exp\left(-\int_0^sF(u)\md u\right)\md s  \le \int_0^t\norm{\be_s}\exp\left(-\int_0^sF(u)\md u\right)\,\md s.
\end{split}
\en
By rearranging the integrating factor, this proves inequality~\eqref{eq:ineq-Y}. Hence, by the Cauchy-Schwartz inequality:
$$
Y^2(t) \le \int_0^t\norm{\be_s}^2\md s \int_0^t\exp\left(2\int_s^tF(u)\md u\right)\md s.
$$
Finally, \eqref{eq:goal} follows by taking $\sup$ over $t \in [0, T]$, and applying expectation $\overline{\EE}$ and ~\eqref{eq:entropy}. 
\end{proof}


\begin{proof}[Proof of Corollary~\ref{cor:RBM}] Follows from Remark~\ref{rmk:strong-existence} and by taking $F\equiv 0$. 
\end{proof}

\begin{proof}[Proof of Theorem~\ref{thm:CBP}] 

\textbf{Proof of (a) for finite systems.} We apply \cite[Proposition 2.11]{LogConcave}. It suffices to show that the drift in the equation~\eqref{eq:mainSDE}:
$$
g(x) = (g_1(x), \ldots, g_N(x))\ \mbox{with}\ \ g_i(t) = \SL_{k=1}^N1(\mP_x(k) = i)g_k \equiv g_{\mP^{-1}_x(i)}
$$
satisfies the contraction condition in~\eqref{eq:contraction}. Rewrite the dot product in~\eqref{eq:contraction} as
\begin{equation}
\label{eq:exposition}
\left(g(x) - g(y)\right)\cdot(x-y) = \SL_{i=1}^Ng_{\mP^{-1}_x(i)}x_i - \SL_{i=1}^Ng_{\mP^{-1}_y(i)}x_i - \SL_{i=1}^Ng_{\mP^{-1}_x(i)}y_i + \SL_{i=1}^Ng_{\mP^{-1}_y(i)}y_i.
\end{equation}
The fact that
\begin{equation}
\label{eq:ineq-exposition}
\SL_{i=1}^Ng_{\mP^{-1}_x(i)}x_i \le \SL_{i=1}^Ng_{\mP^{-1}_y(i)}x_i,\quad\mbox{or, equivalently,}\quad \SL_{k=1}^Ng_kx_{\mP_x(k)} \le \SL_{k=1}^Ng_kx_{\mP_y(k)},
\end{equation}
follows from \cite[Lemma 1.4]{JM2008} applied to $a(i) = -g_i,\ b(i) = x_i,\ i = 1, \ldots, N;\ \tau = \mP_y\mP_x^{-1}$. Similarly,  
$$
\SL_{i=1}^Ng_{\mP^{-1}_y(i)}y_i \le \SL_{i=1}^Ng_{\mP^{-1}_x(i)}y_i.
$$
This proves the contraction condition~\eqref{eq:contraction}, and thus completes the proof. 
 
\textbf{Proof of (b) for finite systems.} The vector $Y = (Y(t), t \ge 0)$ of ranked particles is a (normally) reflected Brownian motion in the (convex) wedge
$
D := \{y \in \BR^N\mid y_1 \le \ldots \le y_N\},
$
with constant drift $g := (g_1, \ldots, g_N)$ and constant diffusion $A = \diag(\si_1^2, \ldots, \si_N^2)$. It suffices to apply Corollary~\ref{cor:RBM} of Theorem~\ref{thm:main}. This completes the proof for finite systems. 

\smallskip

\textbf{The case of infinite systems.} Approximate the infinite system by corresponding finite systems. For every $N \ge 1$, consider a system of $N$ competing Brownian particles 
$
X = (X^{(N)}_1, \ldots, X_N^{(N)})
$
with drift and diffusion coefficients $g_1, \ldots, g_N$ and $\si_1^2, \ldots, \si_N^2$, starting from $(x_1, \ldots, x_N)$. Denote the corresponding ranked particles by 
$
Y_k^{(N)}, \quad k = 1, \ldots, N.
$
By \cite[Theorem 3.3]{MyOwn6}, as $N \to \infty$, via some subsequence we have weak convergence in $C([0, T], \BR^k)$:
\begin{align*}
\bigl(X_1^{(N)}, \ldots, X_k^{(N)}\bigr) & \Ra \bigl(X_1, \ldots, X_k\bigr) \quad \mbox{in case of (a)};\\
\bigl(Y^{(N)}_1, \ldots, Y^{(N)}_k\bigr) &\Ra \bigl(Y_1, \ldots, Y_k\bigr)\quad \mbox{in case of (b)}.
\end{align*}
Since $T_2(C)$ inequalities are preserved under weak limits (\cite[Lemma 2.2]{Djellout}), we are done. 
\end{proof}


\end{document}